\documentclass[11pt]{article}
\usepackage{lmodern}
\usepackage{amsmath}
\usepackage[dvipsnames,usenames]{color}
\usepackage{amsfonts}
\usepackage{mathrsfs}
\usepackage{amssymb}
\usepackage{amsthm}
\usepackage{etoolbox}

\def\<{\langle}
\def\>{\rangle}
\numberwithin{equation}{section}
\newcommand{\Span}{\operatorname{span}}
\def\<{\langle}   
\def\>{\rangle}

\def\-{\overline}

\def\ov{\overline}

\def\CC{{\mathbb{C}}}

\def\ov{\overline}

\def\ld{\lambda}

\def\ord{\hbox{Ord}}

\newtheorem{theorem}{Theorem}[section]
\newtheorem{lemma}[theorem]{Lemma}

\newtheorem{conjecture}[theorem]{Conjecture}
\newtheorem{example}[theorem]{Example}
\newtheorem{remark}[theorem]{Remark}
\usepackage{xcolor}

\colorlet{myrefcolor}{gray!55!blue}
\usepackage[
colorlinks=true,
linkcolor=myrefcolor,
citecolor=myrefcolor,
urlcolor=black
]{hyperref}

\date{\ }	
\begin{document}
\title{\vspace{-1cm} \bf A Hopf Lemma for Holomorphic Maps into Hyperquadrics  \rm}

\author{Xiaojun Huang, \   Yuan Zhang\ \,  and \  Weixia Zhu}%\footnote{partially supported by NSF DMS-1501024}}
\date{}

\maketitle
\begin{center}
\vspace{-8pt} \small \itshape
  Dedicated to Professor Linda Rothschild  on the occasion of her 80th birthday 
\vspace{-5pt} 
\end{center}

\begin{abstract}

More than twenty years ago, Baouendi and the first author proved that a holomorphic map between hyperquadrics of the same signature is either totally degenerate or has a nonvanishing normal derivative for its normal component. This established a CR analogue of the classical Hopf lemma in arbitrary codimension, in the absence of pseudoconvexity. They further conjectured that the same Hopf-type property holds for holomorphic maps between Levi-nondegenerate hypersurfaces of the same signature. In this paper, we provide a counterexample to this  conjecture in full generality. We also prove the conjecture when the target hypersurface is a hyperquadric of any codimension, arguably the most important case for applications. 

\end{abstract}
\renewcommand{\thefootnote}{\fnsymbol{footnote}}
\footnotetext{\hspace*{-7mm}
\begin{tabular}{@{}r@{}p{16.5cm}@{}}
& 2020 Mathematics Subject Classification. Primary 32H02; Secondary  32V30.\\
& Key words and phrases.
CR transversality, Levi-nondegenerate,   Witt decomposition, non-isotropic rescaling. \\
& X. Huang is partially supported by NSF DMS-2247151.\\
& W.\ Zhu is partially supported by  FWF grant 10.55776/PAT1879425.
\end{tabular}}

\section{Introduction}

Let $M_1\subset\CC^n$ and $M_2\subset\CC^N$ be connected smooth CR hypersurfaces, where $2\leq n\leq N$. Let $F$ be a holomorphic map defined in a neighborhood $U\subset\CC^n$ of $M_1$ such that $F(M_1)\subset M_2$. A fundamental problem in the study of holomorphic mappings between real hypersurfaces is to determine when $F$ is CR transversal. More precisely, for $p\in M_1$, we say that $F$ is CR transversal to $M_2$ at $p$ if
\begin{equation*}
T^{1,0}_{F(p)}M_2+dF_p\bigl(T^{1,0}_p\CC^n\bigr)=T^{1,0}_{F(p)}\CC^N.
\end{equation*}
Roughly speaking, CR transversality amounts to the nonvanishing of the normal derivative of the normal component of $F$.

The transversality problem has been extensively studied in the literature, particularly in connection with regularity and rigidity questions for CR mappings in several complex variables. When the target hypersurface is strongly pseudoconvex, CR transversality follows by composing the map with a plurisubharmonic defining function and applying the classical Hopf lemma. For this reason, CR transversality is sometimes referred to as the CR Hopf lemma property. When the map is assumed to be holomorphic only on one side of the source hypersurface, CR transversality is, in many important settings, equivalent to a boundary unique-continuation problem for holomorphic maps \cite{BR90}. This problem has been actively studied from the perspectives of both partial differential equations and several complex variables. Related results may be found, for example, in the work of Huang and Krantz \cite{HK93}, Baouendi--Rothschild \cite{BR93in, BR93, BR93I}, Baouendi--Huang--Rothschild \cite{BHR95}, Berhanu \cite{Be21, Be25}, Berhanu-Hounie \cite{BH}, and others.

In the equidimensional case $n=N$, contributions were made by Forn{\ae}ss \cite{Fo78}, Baouendi--Rothschild \cite{BR90}, Ebenfelt--Rothschild \cite{ER06}, Huang \cite{Hu96}, Isaev \cite{Is88,Is95}, Huang--Pan \cite{HP96}, %Mir \cite{Mi02}, 
%Mir--Lamel \cite{LMir}, and Meylan--Mir--Zaitsev \cite{MMZ03}, 
among others.
For discussions of the connections among unique continuation for solutions of partial differential equations, the boundary Hopf lemma, and CR transversality, we refer the reader to the recent comprehensive survey article by Berhanu \cite{Be21a}.

In higher codimension, Baouendi--Huang \cite{BH05} proved a somewhat surprising rigidity theorem for mappings between hyperquadrics of the same positive signature. Based on this new rigidity, they proved that a holomorphic mapping between hyperquadrics of the same signature is either CR transversal or maps a full neighborhood of the source into the target hyperquadric.  Baouendi--Ebenfelt--Rothschild \cite{BER07} subsequently proved, in a more general setting, that CR transversality holds on an open dense subset under suitable hypotheses. We also mention  the related works of Baouendi--Ebenfelt--Huang \cite{BEH11}, Mir--Lamel \cite{LMir} and  Ebenfelt--Son \cite{ES12}. We  refer the reader to Rothschild \cite{Ro06} for a detailed historical account of such problems and many references. 

The rigidity and transversality theorems of Baouendi--Huang \cite{BH05} led them to pose the following question:

\begin{conjecture}[Baouendi--Huang, 2005]\label{bh}
Let $M_1\subset\CC^n$ and $M_2\subset\CC^N$ be connected, Levi-nondegenerate, real-analytic hypersurfaces of the same signature $\ell$, where $0<\ell\leq(n-1)/2$ and $3\leq n\leq N$. Let $F$ be a holomorphic map defined in a neighborhood of $M_1$ such that $F(M_1)\subset M_2$. Then either $F$ is a local CR embedding of $M_1$ into $M_2$, or $F$ is totally degenerate, in the sense that it maps a neighborhood of $M_1$ in $\CC^n$ into $M_2$.
\end{conjecture}

For Levi-nondegenerate hypersurfaces of the same signature, CR transversality of $F$ at a point is equivalent to $F$ being a local CR embedding near that point. The hyperquadric-target case of the above problem has received particular attention in connection with rigidity and extension problems. Recall that the standard hyperquadric of signature $\ell$ in $\CC^N$ is
\begin{equation}\label{hyq}
\mathbb H^N_\ell=\bigl\{(Z,W)\in\CC^{N-1}\times\CC:\operatorname{Im}W=|Z|_\ell^2\bigr\},
\end{equation}
where $\langle a,\overline b\rangle_\ell=-\sum_{j=1}^\ell a_j\overline b_j+\sum_{j=\ell+1}^{N-1}a_j\overline b_j$ and $|a|_\ell^2=\langle a,\overline a\rangle_\ell$.

The first progress toward the conjecture in this setting was obtained by Huang--Zhang \cite{HZ13,HZ15}, who proved the transversality  for holomorphic mappings $F:M_\ell\to \mathbb H^N_\ell$ when $M_\ell\subset\CC^n$ is a smooth Levi-nondegenerate hypersurface of signature $\ell$ under the codimension restraint $N-n<(n-1)/2$. The proof combines a scaling method, a quantitative version of a lemma of the first author with a detailed normalization and weighted jet analysis. More recently, Huang--Zhu \cite{HZ25} generalized  this result and proved the same conclusion under the larger codimension range $N-n<n-1$. %Their result applies to smooth Levi-nondegenerate source hypersurfaces. 
Thus the remaining issue in the hyperquadric-target case is whether the codimension restriction can be removed altogether. The main result of this paper gives an affirmative answer.

\begin{theorem}\label{main}
Let $M_\ell\subset\CC^n$, $n\geq3$, be a smooth Levi-nondegenerate hypersurface of signature $\ell$, where $0<\ell\leq(n-1)/2$, and let $p\in M_\ell$. Suppose that $F$ is holomorphic in a neighborhood of $p$ and sends $M_\ell$ into $\mathbb H^N_\ell$, $N\geq n$. Then either $F$ is CR transversal at $p$, or $F$ maps a neighborhood of $p$ in $\CC^n$ into $\mathbb H^N_\ell$. Equivalently, if $F$ is not totally degenerate as a germ at $p$, then it is a local CR embedding near $p$.
\end{theorem}

Here the last assertion follows from the standard Levi-form identity for a CR transversal map between Levi-nondegenerate hypersurfaces of the same signature: transversality implies that the differential is injective on the CR tangent space, and hence $F$ is a local CR embedding. 
%This theorem gives an affirmative answer to the Baouendi--Huang transversality problem when the target is a hyperquadric of the same signature. 
The hyperquadric structure of the target enters Theorem \ref{main} in an essential way. Indeed, {\it we will show that the transversality asserted in Conjecture \ref{bh} fails when the target is an arbitrary Levi-nondegenerate hypersurface of the same signature}, and Theorem \ref{main} is sharp in this respect; see Example \ref{ex:counterexample}. 

We briefly describe the main idea of the proof. The basic tool remains the scaling method, as in \cite{HZ15,HZ25}. However, the scaling used here differs fundamentally from that in \cite{HZ15,HZ25}. In those works, we used the standard Heisenberg scaling, which is isotropic and elliptic in the complex tangential directions. To obtain convergence under such a scaling, one needs very precise normalizations and quantitative jet estimates, relying essentially on the results of Mir, Meylan-Mir-Zaitsev and  Ebenfelt-Huang-Zaitsev \cite{Mi02,MMZ03,EHZ04}.

In the present work, we use a substantially different hyperbolic non-isotropic scaling  combined with the Heisenberg dilations. At a nontransversal point, the mapping equation takes the form
$\rho_{{\mathbb H}_\ell^N}\circ F=a\rho_{M_\ell}$, where the $\rho$'s are defining functions and $a$ is real-analytic and vanishes at the point under consideration. Assuming that $F$ is not totally degenerate, we examine the first nonzero weighted term of $a$. The corresponding low-order terms of the tangential component of $F$ satisfy a family of orthogonality relations with respect to the target Levi form. These
relations yield an adapted Witt decomposition of the target space associated with a totally isotropic subspace.

Moreover, the dimension of this isotropic subspace is bounded by the Levi
signature $\ell$, and is therefore independent of the codimension $N-n$. The key new idea of this paper is to combine hyperbolic non-isotropic dilations in $U(N-1,\ell)$, arising from the Witt decomposition associated with $F$, with Heisenberg dilations. This enables us to rescale $F$ and obtain a limit map
\[
F_0\colon {\mathbb H}_\ell^n\to {\mathbb H}_\ell^N
\]
that is not CR transversal at the point under consideration. The
Baouendi--Huang theorem then implies that $F_0$ must be totally degenerate. Finally, the limiting mapping equation implies that $F$ itself is totally degenerate.

Non-isotropic dilations in CR-tangential directions have previously been used by several authors; see, for example, Huang and Yin \cite{HY09}, and, subsequently, the first author with James and Li \cite{HJL26}. Geometrically, non-isotropic scaling may be viewed as scaling along tangential orbits, and has proved to be a powerful tool. However, the non-isotropic dilations considered in \cite{HY09,HJL26} are not hyperbolic, since the settings in those works are pseudoconvex. In a recent elegant paper, Fang \cite{Fa26} successfully used hyperbolic non-isotropic dilations in $SO(n,\mathbb{C})$, as suggested by the
first author, to resolve an old problem of Morimoto and Nagano.

\section{A counterexample to the original conjecture}

In this section, we construct the following example 
in which $M_1$ and $M_2$ are real-analytic Levi-nondegenerate
hypersurfaces of the same signature $\ell=1$, while the holomorphic mapping $F$ is neither CR transversal at the origin nor totally degenerate.
Thus the corresponding transversality statement fails for a general
Levi-nondegenerate target hypersurface.  In particular, this example shows that the hyperquadric structure of the target is essential in Theorem~\ref{main}.

\begin{example}\label{ex:counterexample}
	
\medskip
Let $M_1=\{(z,w)\in  \mathbb C^4: \ \rho_1={\rm{Im}}\  w+|z_1|^2-|z_2|^2-|z_3|^2=0\}$,  and  $ M_2=\{(Z,W)\in {\mathbb C^5}:\ \rho_2={\rm{Im}}\ W+|Z_1|^2-|Z_2|^2-|Z_3|^2-|Z_4|^2+2|Z_1|^2|Z_2|^2=0 \}.$
For $\ld(\neq 0)\in {\mathbb C}$, we set $f_1=\dfrac{z_1}{\sqrt 2(z_2+\ld)}$. Consider the holomorphic map $F$, defined in a neighborhood of $0\in \mathbb C^4$,  by
$$
F=\bigg(f_1, \dfrac{2(z_2+\ld)}{1-iw+2z_2\overline \ld+|\ld|^2}f_1,\dfrac{2z_3}{1-iw+2z_2\overline \ld+|\ld|^2}f_1,\dfrac{1+iw-2z_2\overline \ld-|\ld|^2}{1-iw+2z_2\overline \ld+|\ld|^2}f_1,0\bigg). 
$$
Then a direct computation gives
$$
\rho_2\circ F=\dfrac{2|z_1|^2}{|z_2+\ld|^2\big|1-iw+2z_2\overline \ld+|\ld|^2\big |^2}\cdot\rho_1.
$$
Hence $F(M_1)\subset M_2$, while $F$ is neither CR transversal at $0$ nor totally degenerate.
\end{example}

\section{Witt decomposition for an indefinite Hermitian form}

In this section, we collect several elementary facts, especially the Witt decomposition, concerning indefinite Hermitian spaces that will be used in the proof. (See, for instance,  \cite[Chapter XV]{La02}\cite[Chapter 7]{Sc85} for the Witt decomposition.) 

Let $\mathbb C^{N-1}$  be equipped with the non-degenerate Hermitian form
$$
\langle Z,\overline{\Xi}\rangle_\ell
=-\sum_{j=1}^{\ell}Z_j\overline{\Xi_j}+
\sum_{j=\ell+1}^{N-1}Z_j\overline{\Xi_j},
$$
of signature $(\ell,N-1-\ell)$. 
A vector $v\in\mathbb C^{N-1}\setminus\{0\}$ is called \emph{isotropic} if $\langle v,\bar v\rangle_\ell=0.$ A complex linear subspace $W\subset\mathbb C^{N-1}$ is 
\emph{totally isotropic} if
$\langle v,\bar w\rangle_\ell=0$ for all $v,w\in W$.
In particular, every nonzero vector in a totally isotropic subspace is isotropic. 

Notice that the elementary dimension bound for a totally isotropic subspace $W$ is $
\dim_{\mathbb C}W\le\min\{\ell,N-1-\ell\}
$ (see for example \cite{La02, Sc85}). We now introduce a form of the Witt decomposition adapted to a totally isotropic subspace \cite{Sc85}. This formulation is  important for us to set up  the non-isotropic  rescaling argument used in our proof. For subspaces $V,W\subset\CC^{N-1}$, we write $V\perp_\ell W$ if they are orthogonal with respect to $\langle\cdot,\overline{\cdot}\rangle_\ell$. 
A subspace $E \subset\mathbb C^{N-1}$   is said to be \emph{nondegenerate} if the restriction of the Hermitian form  $ \langle \cdot, \bar\cdot\rangle_\ell$ to
$E$ is nondegenerate, or equivalently,
$ 
E\cap E^{\perp_\ell}=\{0\},
$ 
where $E^{\perp_\ell}$ is the orthogonal complement of $E$ with respect to  $ \langle \cdot, \bar\cdot\rangle_\ell$.

\begin{lemma}\label{witt}
Let $W\subset\CC^{N-1}$ be a totally isotropic subspace of dimension $q\geq1$, and let $e_1,\ldots,e_q$ be a basis of $W$. There are vectors $e'_1,\ldots,e'_q$ such that
\begin{equation}\label{w1}
\langle e_\alpha,\overline{e'_\beta}\rangle_\ell=\delta_{\alpha\beta},\qquad \langle e'_\alpha,\overline{e'_\beta}\rangle_\ell=0,
\end{equation}
for $1\leq\alpha,\beta\leq q$. Setting $W'=\Span\{e'_1,\ldots,e'_q\}$ and $E=(W\oplus W')^{\perp_\ell}$, then
\begin{equation}\label{w2}
\CC^{N-1}=W\oplus W'\oplus E,
\end{equation}
where $E$ is nondegenerate.
\end{lemma}

\begin{proof}
For convenience of the reader, we sketch a proof of this  lemma \cite{La02, Sc85} as follows:
Choose  a linearly independent set $u_1,\ldots,u_q$ with $\langle e_\alpha,\bar u_\beta\rangle_\ell=\delta_{\alpha\beta}$, and set $H_{\alpha\ov\beta}=\langle u_\alpha,\bar
u_\beta\rangle_\ell$.  
Define  $e'_\alpha=u_\alpha-\frac12\sum_{\beta=1}^{q}H_{\alpha\ov\beta}e_\beta$, 
one verifies that 
$\langle e_\alpha',\bar e'_\gamma\rangle_\ell = 0$. Hence $W'$ is
totally isotropic.
The sum $W+W'$ is direct: if $v=\sum_\alpha c_\alpha e_\alpha=\sum_\beta d_\beta
e'_\beta$, then pairing with $\bar e'_\gamma$ gives $c_\gamma=0$ for all $\gamma$
because $W'$ is totally isotropic.  In the basis $(e_1,\dots,e_q,e'_1,\dots,e'_q)$,
the  matrix of $\langle\cdot,\bar\cdot\rangle_\ell$ restricted to $W\oplus W'$
is $\left(\begin{smallmatrix}0&I_q\\ I_q&0\end{smallmatrix}\right)$, which is
invertible; hence $W\oplus W'$ is nondegenerate.   Its orthogonal complement $E$ is therefore nondegenerate, and \eqref{w2} follows.
\end{proof}

With respect to \eqref{w2}, every vector $Z\in\CC^{N-1}$ has a unique decomposition 
$$
Z=\sum_{\alpha=1}^q x_\alpha e_\alpha+\sum_{\alpha=1}^q y_\alpha e'_\alpha+Z_E,
$$
where $Z_E\in E$. Moreover, it follows from \eqref{w1} that $x_\alpha=\langle Z,\overline{e'_\alpha}\rangle_\ell$, $y_\alpha=\langle Z,\overline e_\alpha\rangle_\ell$, and $\langle Z,\overline Z\rangle_\ell=\sum_{\alpha=1}^q(x_\alpha\overline y_\alpha+y_\alpha\overline x_\alpha)+\langle Z_E,\overline{Z_E}\rangle_\ell$.

Denote by $U(N-1,  \ell)$ the indefinite unitary group preserving the
Hermitian form $\langle\cdot,\bar\cdot\rangle_\ell$ on
$\mathbb C^{N-1}$. The following classical hyperbolic non-isotropic
dilation in $U(N-1, \ell)$, expressed in Witt coordinates
(see \cite{He01}), will play a fundamental role in our proof.
Related non-isotropic dilations have also been used in other  contexts;
see, for example, Gao--Ng--Seo \cite{GNS} and  Broussous--Stevens \cite{BroussousStevens}, etc.
\begin{lemma}\label{dil}
Let $\lambda_1,\ldots,\lambda_q>0$. Define a linear transformation $\Lambda:\mathbb C^{N-1}\to\mathbb C^{N-1}$ by $\Lambda e_\alpha=\lambda_\alpha e_\alpha$, $\Lambda e'_\alpha=\lambda_\alpha^{-1}e'_\alpha$, and $\Lambda|_E=I$. Then $\Lambda\in U(N-1, \ell)$.
\end{lemma}

\begin{proof}
Let
\[
Z=\sum_{\alpha=1}^q x_\alpha e_\alpha+\sum_{\alpha=1}^q y_\alpha e'_\alpha+Z_E,
\qquad
\Xi=\sum_{\alpha=1}^q \xi_\alpha e_\alpha+\sum_{\alpha=1}^q \eta_\alpha e'_\alpha+\Xi_E,
\]
with $Z_E,\Xi_E\in E$. Then
$\langle Z,\overline{\Xi}\rangle_\ell
=\sum_{\alpha=1}^q\left(x_\alpha\overline{\eta_\alpha}+y_\alpha\overline{\xi_\alpha}
\right)+\langle Z_E,\overline{\Xi_E}\rangle_\ell.$
Since $\lambda_\alpha>0$, we have
\[
\begin{aligned}
\langle \Lambda Z,\overline{\Lambda \Xi}\rangle_\ell
&=\sum_{\alpha=1}^q
\left(
\lambda_\alpha x_\alpha
\overline{\lambda_\alpha^{-1}\eta_\alpha}
+\lambda_\alpha^{-1}y_\alpha
\overline{\lambda_\alpha\xi_\alpha}
\right)
+\langle Z_E,\overline{\Xi_E}\rangle_\ell =
\langle Z,\overline{\Xi}\rangle_\ell.
\end{aligned}
\]
Thus $\Lambda$ preserves the Hermitian form. 
\end{proof}

\section{Proof of the main theorem}

In the proof of the main theorem, the subspace $W$ will arise from the coefficient vectors of the low-weighted terms in the tangential component of the mapping. The orthogonality relations among these terms imply that $W$ is totally isotropic. We then choose a basis $\{e_\alpha\}$ of $W$ adapted to the weighted filtration and apply a hyperbolic, non-isotropic rescaling composed with the Heisenberg scaling to reduce to mappings between hyperquadrics in \cite{BH05}.
We assume, after replacing $\rho_{M_\ell}$ by $-\rho_{M_\ell}$ and permuting the $z_j$ if necessary, that the number of negative Levi eigenvalues equals $\ell\le (n-1)/2$.

\begin{proof}[Proof of Theorem \ref{main}:]
We first consider the case when  $M_\ell\subset \mathbb C^n$ is a real-analytic Levi-nondegene-rate
hypersurface of signature $\ell$, where $ 0<\ell\le \frac{n-1}{2}.$ 
After choosing normal coordinates at the origin (see, for example,  \cite{HZ15}) and assigning the usual weights ${\rm wt} (z_j) =1,  {\rm wt} (w)=2$, we write a real analytic defining equation of $M_\ell$ to be 
$$ 
\rho_{M_\ell}(z,w,\bar z,\bar w)
={\rm{Im}}\ w 
-\langle z,\bar z\rangle_\ell
+O_{\mathrm{wt}}(4) 
$$ 
near the origin, where $\langle z,\bar z\rangle_\ell
=-\sum_{j=1}^{\ell}z_j\bar z_j+\sum_{j=\ell+1}^{n-1}z_j\bar z_j$. 
Since the group $\operatorname{Aut}(\mathbb H^N_\ell)$ acts transitively on $\mathbb H^N_\ell$, after composing $F$ with a suitable $\sigma\in\operatorname{Aut}(\mathbb H^N_\ell)$, we may further assume that $F(0)=0$. Thus $F=(f,g):(\mathbb C^n,0)\to (\mathbb C^N,0)$ is a germ of a holomorphic mapping such that $F(M_\ell)\subset {\mathbb H}_\ell^N$. For the target ${\mathbb H}_\ell^N$, it has defining function \eqref{hyq}. 
Then there exists a real-analytic function $a$ such that
\begin{equation}\label{eq:factor}
\rho_H(F(z,w),\overline{F(z,w)})
=a(z,w,\bar z,\bar w)\rho_{M_\ell}(z,w,\bar z,\bar w).
\end{equation}
Moreover, the mapping $F$ is CR transversal at the origin precisely when $a(0)\ne0$.  

Supposing that $a(0)=0$, we seek to show that $a\equiv0$. Namely, $F$ maps a full neighborhood of the origin into ${\mathbb H}_\ell^N$. Assume by contradiction that $a\not\equiv 0$. Let $r:=\ord_{\mathrm{wt}} a \ge1$ and write
$$
a=a^{(r)}+O_{\mathrm{wt}}(r+1),
\qquad
a^{(r)}\not\equiv0.
$$
Here, $a^{(r)}$ is a weighted homogeneous polynomial of weighted order $r$.

We first complexify \eqref{eq:factor}. Let $(z,w)$ and $(\chi,\tau)$ be independent variables in $\mathbb C^{n-1}\times\mathbb C$, with $\chi$ and $\tau$ assigned weights $1$ and $2$, respectively. We write $\bar f(\chi,\tau)$ and $\bar g(\chi,\tau)$ for the holomorphic functions obtained by conjugating the coefficients of $f$ and $g$. Similarly, let $A$ and $R$ be the complexifications of $a$ and $\rho_{M_\ell}$. Then
\begin{equation}\label{eq:complexified}
\frac{g(z,w)-\bar g(\chi,\tau)}{2i}-\left\langle
f(z,w),\bar f(\chi,\tau)\right\rangle_\ell
=A(z,w,\chi,\tau)R(z,w,\chi,\tau).
\end{equation}
Here
$$
R(z,w,\chi,\tau)=\frac{w-\tau}{2i}-\langle z,\chi\rangle_\ell+O_{\mathrm{wt}}(4).
$$
Since $\ord_{\mathrm{wt}}A=r$ and $\ord_{\mathrm{wt}}R=2$, the right-hand side of \eqref{eq:complexified} has weighted order at least $r+2$, which implies
\begin{equation}\label{eq:LHSvanish}
\left[\frac{g(z,w)-\bar g(\chi,\tau)}{2i}-
\left\langle f(z,w),\bar f(\chi,\tau)\right
\rangle_\ell \right]^{(<{r+2})}=0,
\end{equation}
where the superscript $(<{r+2})$ denotes the sum of all monomial terms of total weighted degree strictly less than ${r+2}$.

Write $f=\sum_{\nu\geq1}f^{(\nu)}$ and $g=\sum_{\nu\geq1}g^{(\nu)}$, where $f^{(\nu)}$ and $g^{(\nu)}$ are homogeneous of weighted  degree $\nu$.
Since $F(0)=0$, the Hermitian product
$\langle f(z,w),\bar f(\chi,\tau)\rangle_\ell$ 
contains no purely $(z,w)$ or $(\chi,\tau)$ terms. 
Taking in \eqref{eq:LHSvanish} the terms independent of
$(\chi,\tau)$, we get
\begin{equation}\label{eq:gvanish}
g^{(\nu)}=0,\qquad\nu<{r+2}.
\end{equation}
Fix integers $i,j\ge1$ with $i+j<{r+2}$. 
Consider in \eqref{eq:LHSvanish} the component having  
weighted degree $i$ in $(z,w)$ and   weighted degree $j$ in $(\chi,\tau)$. It follows from \eqref{eq:gvanish} that
\begin{equation}\label{eq:orthogonal-polynomials}
 \left\langle
 f^{(i)}(z,w),
 \bar f^{(j)}(\chi,\tau)
 \right\rangle_\ell
 =0,
 \qquad
 i+j<{r+2}.
\end{equation}
For each $i\ge 1$, write 
$$
f^{(i)}(z,w)=\sum_{|\alpha|+2p=i}
v_{\alpha p}z^\alpha w^p.
$$
From \eqref{eq:orthogonal-polynomials}, we have
$$
\sum_{\substack{|\alpha|+2p=i,\\|\beta|+2q=j}}
\left\langle v_{\alpha p},\overline{v_{\beta q}}\right\rangle_\ell z^\alpha w^p\chi^\beta\tau^q
=0,  
\qquad i+j<{r+2}.
$$
Since the monomials $z^\alpha w^p\chi^\beta\tau^q$ are linearly independent, we obtain
\begin{equation}\label{eq:coefficient-orthogonality}
\left\langle v_{\alpha p},\overline{v_{\beta q}}\right\rangle_\ell=0 
\qquad \text{whenever}\ \ |\alpha|+2p+|\beta|+2q<{r+2}.
\end{equation}
For each $j\ge1$, let $V_j\subset\mathbb C^{N-1}$ denote the complex linear span of all coefficient vectors occurring in $f^{(j)}$. Equation \eqref{eq:coefficient-orthogonality} says that
\begin{equation}\label{eq:Vorthogonal}
V_i\perp_\ell {V_j}
\quad \text{whenever}\ \  i+j<{r+2}.
\end{equation}
Set $W=\sum_{j<(r+2)/2}V_j$. Then $W$ is a totally isotropic subspace of $\mathbb C^{N-1}$ with the Hermitian form  $ \langle \cdot, \bar\cdot\rangle_\ell$ of signature $\ell$. Note that $q:=\dim_{\mathbb C}W\le\ell.$ 

For each integer $s<(r+2)/2$, set $W_s:=\sum_{j\le s}V_j$. Choose a basis $e_1,\ldots,e_q$ of $W$ adapted to this filtration. Thus for each $\alpha\in\{1,\cdots,q\}$, there exists an integer $\mu_\alpha<(r+2)/2$, such that
\begin{equation}\label{eq:adapted}
\mu_\alpha=\min\{s:e_\alpha\in W_s\},\qquad W_s=\Span\{e_\alpha:\mu_\alpha\leq s\}.
\end{equation}
If $W=0$, all sums below are empty, $E=\mathbb C^{N-1}$, and $\Lambda_t=I$. Otherwise, apply Lemma \ref{witt} and write $\mathbb C^{N-1}=W\oplus W'\oplus E$, with dual isotropic basis $e'_1,\ldots,e'_q$. Decompose
\begin{equation}\label{eq:fdecompose}
f=\sum_{\alpha=1}^q x_\alpha e_\alpha+\sum_{\alpha=1}^q y_\alpha e'_\alpha+h,
\qquad h\in E,  
\end{equation}
where $x_\alpha, y_\alpha$ are scalar valued holomorphic functions while $h$ is an $E$-valued holomorphic mapping, with
\begin{equation}\label{eq:yal}
x_\alpha=\langle f,\overline{e'_\alpha}\rangle_\ell,\qquad y_\alpha=\langle f,\overline{e_\alpha}\rangle_\ell, 
\qquad \alpha=1,\ldots,q. 
\end{equation}
 
We now estimate the orders of the three parts in \eqref{eq:fdecompose}. If $j<\mu_\alpha$, then $V_j\subset W_{\mu_\alpha-1}$ and the adapted dual basis gives $\langle V_j,\overline{e'_\alpha}\rangle_\ell=0$. Hence $\ord_{\mathrm{wt}} x_\alpha\geq\mu_\alpha$.
Since $V_j\subset W$ for $j<(r+2)/2$, the $E$-component of
$f^{(j)}$ vanishes for all such $j$. %we see that $f-\sum_{0<j< (r+2)/2}f^{(j)}=h \  \hbox{mod} \ W\oplus W'$ by  (\ref{eq:fdecompose}). Differentiating on both sides and then taking $(z,w)=0$, we obtain
In view of \eqref{eq:fdecompose}, it follows that the $E$-component satisfies $\ord_{\mathrm{wt}}h\geq\lceil(r+2)/2\rceil$, where $\lceil x \rceil $ denotes the smallest integer greater than or equal to $x$. 
Moreover, $e_\alpha\in W_{\mu_\alpha}$ and \eqref{eq:Vorthogonal} implies $V_j\perp_\ell W_{\mu_\alpha}$ when $j+\mu_\alpha<r+2$. Thus $\ord_{\mathrm{wt}}y_\alpha\geq r+2-\mu_\alpha$. Altogether, we have proved
\begin{equation}\label{ord}
\ord_{\mathrm{wt}}x_\alpha\geq\mu_\alpha,\qquad \ord_{\mathrm{wt}}y_\alpha\geq r+2-\mu_\alpha,\qquad \ord_{\mathrm{wt}}h\geq\lceil \frac{r+2}{2}\rceil.
\end{equation}

The point of \eqref{ord} is that the two isotropic directions occur at complementary weights. They can therefore be rescaled by reciprocal powers without changing the target Levi form. 
For $t>0$, define a hyperbolic non-isotropic scaling transformation $\Lambda_t$, adapted to this  Witt decomposition, of $\mathbb C^{N-1}$ by
\begin{equation}\label{eq:Ut}
\Lambda_te_\alpha
=t^{\frac{r+2}{2}-\mu_\alpha}e_\alpha,
\qquad \Lambda_t e'_\alpha
=t^{-(\frac{r+2}{2}-\mu_\alpha)}e'_\alpha,
\qquad\Lambda_t|_E=I.
\end{equation}
Let $\delta_t(z,w)=(tz,t^2w)$ and set
\begin{equation}\label{eq:Ft}
f_t(z,w):=t^{-{\frac{r+2}{2}}}\Lambda_tf(tz,t^2w),
\qquad
g_t(z,w):=t^{-({r+2})}g(tz,t^2w).
\end{equation}

We claim that $F_t=(f_t,g_t)$ converges locally uniformly as $t\to0^+$. Indeed, from \eqref{eq:fdecompose} and \eqref{eq:Ut}, we have
$$
f_t=\sum_{\alpha=1}^q t^{-\mu_\alpha}x_\alpha(\delta_t)e_\alpha
+\sum_{\alpha=1}^q t^{-({r+2}-\mu_\alpha)}
y_\alpha(\delta_t)e'_\alpha
+t^{-{\frac{r+2}{2}}}h(\delta_t).
$$
By \eqref{ord}, $$t^{-\mu_\alpha}x_\alpha(\delta_t)
\to x_\alpha^{(\mu_\alpha)},\qquad 
t^{-({r+2}-\mu_\alpha)}y_\alpha(\delta_t)\to y_\alpha^{({r+2}-\mu_\alpha)}$$
locally uniformly. Moreover,
$$
t^{-{\frac{r+2}{2}}}h(\delta_t)\to
\begin{cases}
 h^{({\frac{r+2}{2}})},& {\frac{r+2}{2}}\in\mathbb N,\\ 
\ \  0,& {\frac{r+2}{2}}\notin\mathbb N.
\end{cases}
$$
On the other hand, \eqref{eq:gvanish} gives
$ 
\ord_{\mathrm{wt}}g\ge {r+2},
$ 
and hence
$$
t^{-({r+2})}g(\delta_t)
\to
g^{({r+2})}
$$
locally uniformly. Therefore
$F_t=(f_t,g_t)$ converges locally uniformly to a holomorphic polynomial map $F_0=(f_0,g_0)$, where
\begin{equation*}\label{eq:F0}
\begin{aligned}
f_0&=\sum_{\alpha=1}^q x_\alpha^{(\mu_\alpha)}e_\alpha+\sum_{\alpha=1}^q y_\alpha^{(r+2-\mu_\alpha)}e'_\alpha+\begin{cases}h^{((r+2)/2)},& \frac{r+2}{2}\in\mathbb N,
\\0,&\frac{r+2}{2}\notin\mathbb N,\end{cases}\\
g_0&=g^{(r+2)}.
\end{aligned}
\end{equation*}

We next compute $\rho_{{\mathbb H}_\ell^N}\circ F_t$.   
Since $\Lambda_t\in \Lambda(\ell,N-1)$ preserves the target Hermitian form, \eqref{eq:Ft} gives the exact identity
\begin{equation}\label{eq:scaling}
\begin{split}
\rho_{{\mathbb H}_\ell^N}\circ F_t
&=\operatorname{Im}\bigl(t^{-(r+2)}g\circ\delta_t\bigr)-\Bigl|t^{-\frac{r+2}{2}}\Lambda _t(f\circ\delta_t)\Bigr|^{2}_{\ell}  \\
&=t^{-(r+2)}\Bigl(\operatorname{Im}(g\circ\delta_t)-|f\circ\delta_t|^2_\ell\Bigr)
=t^{-(r+2)}\rho_{{\mathbb H}_\ell^N}(F\circ\delta_t).
\end{split}
\end{equation}
By \eqref{eq:factor}, we have
\begin{equation}\label{eq:limit-factor}
\rho_{{\mathbb H}_\ell^N}\circ F_t
=t^{-({r+2})}a(\delta_t)\rho_{M_\ell}(\delta_t) =\left(t^{-r}a(\delta_t)\right)
\left(t^{-2}\rho_{M_\ell}(\delta_t)\right).
\end{equation}
Since 
$t^{-r}a(\delta_t)\to a^{(r)}$
and $t^{-2}\rho_{M_\ell}(\delta_t)
\to\rho_{{\mathbb H}_\ell^n}$ as $t\to0^+$, where $\rho_{{\mathbb H}_\ell^n}(Z,W) = {\rm{Im}}\ w-\langle z,\bar z\rangle_\ell$.
Passing to the limit in \eqref{eq:limit-factor} gives
\begin{equation}\label{eq:keylimit}
\rho_{{\mathbb H}_\ell^N}(F_0)=a^{(r)}\rho_{{\mathbb H}_\ell^n}.
\end{equation}
In particular,
$F_0({\mathbb H}_\ell^n)\subset {\mathbb H}_\ell^N$.  On the other hand,  since $g_0=g^{(r+2)}$ and $r+2\geq 3$, no term of $g_0$ has weighted degree two. Since $w$ has weighted degree two, we have $(g_0)_w(0)=0$.
We now apply the arbitrary-codimension rigidity theorem of Baouendi--Huang \cite[Theorem~1.6(ii)]{BH05} for mappings between hyperquadrics of the same signature to $F_0$. It follows that $\rho_{{\mathbb H}_\ell^N}\circ F_0\equiv0$.
Equation \eqref{eq:keylimit} then gives $a^{(r)}\rho_{{\mathbb H}_\ell^n}\equiv0$.
Since $\rho_{{\mathbb H}_\ell^n}$ is not identically zero, $a^{(r)}\equiv0$, a contradiction. Hence $a\equiv0$, and the theorem is proved when $M_\ell$ is real analytic.

Now suppose that $M_\ell$ is only smooth. If $F$ is not totally degenerate as a germ at $p$, then $\rho_H\circ F$ is a nonzero real-analytic function germ. Its zero set $X=\{\rho_H\circ F=0\}$ is a real-analytic variety with $\dim_\mathbb R X_p\leq2n-1$. Since $M_\ell\subset X$, we also have $\dim_\mathbb R X_p\geq2n-1$. Thus $\dim_\mathbb R X_p=\dim_\mathbb R M_\ell=2n-1$. By Malgrange's theorem \cite[Proposition 3.11]{Ma67}, the germ of $M_\ell$ at $p$ is real analytic. The result just proved applies and completes the proof.
%Suppose now that ${M_\ell}\subset\mathbb C^n$ is only $C^\infty$, Levi-nondegenerate of signature $\ell$, and that $  F:(\mathbb C^n,p)\to \mathbb C^N $  is holomorphic near $p$ with $  F({M_\ell})\subset {\mathbb H}_\ell^N. $  Assuming that $F$ is not totally degenerate near $p$, it follows that the real-analytic function $  {\rho_{{\mathbb H}_\ell^N}\circ F}\not\equiv0  $,  and its zero set $X:\  =\{\rho_{{\mathbb H}_\ell^N}\circ F=0\}$ is a real-analytic variety satisfying $  \dim_{\mathbb R}X_p\le 2n-1. $ On the other hand, since $ {M_\ell}\subset X$,   this inclusion forces   $$ \dim_{\mathbb R}X_p\ge \dim_{\mathbb R}M_\ell  =  2n-1. $$   Consequently, $$ \dim_{\mathbb R}X_p=2n-1 = \dim_{\mathbb R}{M_\ell}. $$ By Malgrange's theorem in \cite{Mal} on smooth submanifolds contained in real-analytic varieties of the same dimension, it follows that the germ of ${M_\ell}$ at $p$ is real-analytic. Therefore the real-analytic result proved previously applies to ${M_\ell}$ at $p$ and completes the proof. 
\end{proof}  
 
We conclude the paper with  a remark    concerning the proof of Theorem \ref{main}.

% \begin{remark}
 %    Although the bound \eqref{eq:dimW} is not used for the later rescaling argument itself, 
%it reflects the intrinsic geometry of the indefinite Levi form and emphasizes that the isotropic part is controlled by the signature rather than by the codimension.
 %\end{remark}

\begin{remark}
The hyperquadric structure of the target is essential for the preceding rescaling argument. Indeed, for a general real-analytic Levi-nondegenerate target written in Chern-Moser normal coordinates as
$$
\rho_2(Z,W)
=\operatorname{Im} W-|Z|_\ell^2
+\Phi(Z,\overline Z,\operatorname{Re}W),
\qquad \Phi=O_{\mathrm{wt}}(4),
$$
the hyperbolic non-isotropic dilation $\Lambda_t$ may magnify certain terms in $\Phi$.  The same conclusion nevertheless holds when $\Phi(Z,\overline Z,\operatorname{Re}W)
=\varphi\bigl(|Z|_\ell^2,\operatorname{Re}W\bigr)$, where $\varphi(s,u)$ is real-analytic near the origin and vanishes to order at least two.
In fact, the preceding argument gives
$\operatorname{Ord}_{\mathrm{wt}}g\ge r+2$ and
$\operatorname{Ord}_{\mathrm{wt}}|f|_\ell^2\ge r+2$.
Since $\varphi(s,u)=O((|s|+|u|)^2)$, it follows that
$\operatorname{Ord}_{\mathrm{wt}}\varphi\bigl(|f|_\ell^2,\operatorname{Re}g\bigr)\ge 2(r+2)$, and hence this term vanishes in the rescaling limit. Thus the limiting identity is again \eqref{eq:keylimit}, and the preceding proof applies.
%Thus an analogue of the exact scaling identity \eqref{eq:scaling} need not hold.  On the other hand, transversality at $0$ still holds when 
%$\Phi(Z,\overline{Z},\operatorname{Re}W)=\varphi\bigl(|Z|_\ell^2,\operatorname{Re} W\bigr)$.
Thus, the $U(N-1,\ell)$-symmetry of the target is crucial for establishing transversality.
\end{remark}

\vspace{0.4cm}

\fontsize{11}{11}\selectfont
\noindent X. Huang. \href{mailto:huangx@math.rutgers.edu}
{\nolinkurl{huangx@math.rutgers.edu}}

\vspace{0.1 cm}

\noindent Department of Mathematics, Rutgers University, New Brunswick, NJ 08903, USA.

\vspace{0.4cm}

\noindent Y. Zhang. \href{zhan1313@pfw.edu}
{\nolinkurl{zhan1313@pfw.edu}}

\vspace{0.1 cm}

\noindent Department of Mathematical Sciences, Purdue University Fort Wayne, Fort Wayne, IN 46805-1499, USA. 

\vspace{0.4cm}

\noindent W. Zhu. \href{wxzhu@xmu.edu.cn; weixia.zhu@univie.ac.at}
{\nolinkurl{wxzhu@xmu.edu.cn};\hspace{0.5em}\nolinkurl{weixia.zhu@univie.ac.at}}
\vspace{0.1cm}

\noindent  
School of Mathematical Sciences, Xiamen University, Xiamen 361005, China;
\vspace{0.1cm}

\noindent
Faculty of Mathematics, University of Vienna, Oskar-Morgenstern-Platz 1, 1090 Vienna, Austria.

\vspace{0.2 cm}
\end{document}